\documentclass{amsproc}
\usepackage{eurosym}
\usepackage{amssymb}
\usepackage{amsfonts}
\usepackage{cmtt}

\theoremstyle{plain}
\newtheorem{theorem}{Theorem}[section]

\newtheorem{question}[theorem]{Question}

\newtheorem{conjecture}[theorem]{Conjecture}
\newtheorem{corollary}[theorem]{Corollary}
\newtheorem{observation}[theorem]{Observation}

\newtheorem{lemma}[theorem]{Lemma}
\newtheorem{proposition}[theorem]{Proposition}
\newtheorem{remark}[theorem]{Remark}
\newtheorem{innercustomtheorem}{Theorem}
\newenvironment{customtheorem}[1]
{\renewcommand\theinnercustomtheorem{#1}\innercustomtheorem}
{\endinnercustomtheorem}
\newtheorem{innercustomcorollary}{Corollary}
\newenvironment{customcorollary}[1]
{\renewcommand\theinnercustomcorollary{#1}\innercustomcorollary}
{\endinnercustomcorollary}

\newcommand{\field}[1]{\mathbb{#1}}

\newcommand{\N}{\field{N}}

\newcommand{\Z}{\field{Z}}
\newcommand{\K}{\field{K}}

\begin{document}

\title{On the toric ideals of matroids of a fixed rank}

\author[Micha\l\ Laso\'{n}]{Micha\l\ Laso\'{n}}

\dedicatory{\upshape
Institute of Mathematics of the Polish Academy of Sciences,\\ ul.\'{S}niadeckich 8, 00-656 Warszawa, Poland\\ \textmtt{michalason@gmail.com}}

\thanks{Research supported by Polish National Science Centre grant no. 2015/19/D/ST1/01180. The paper was initiated during author's stay at Freie Universit\"at Berlin in the frame of Polish Ministry ``Mobilno\'s\'c Plus'' program.}
\keywords{Matroid, toric ideal, base exchange, generating set, Gr\"{o}bner basis, Betti table}

\begin{abstract}
In $1980$ White conjectured that every element of the toric ideal of a matroid is generated by quadratic binomials corresponding to symmetric exchanges.

We prove White's conjecture for high degrees with respect to the rank. This extends our result \cite{LaMi14} confirming White's conjecture `up to saturation'. 
Furthermore, we study degrees of Gr\"{o}bner bases and Betti tables of the toric ideals of matroids of a fixed rank.
\end{abstract}

\maketitle

\section{Introduction}

\subsection{Problem description}

Let $M$ be a matroid on the ground set $E$ with the set of bases $\mathfrak{B}$ and the rank function $r:\mathcal{P}(E)\rightarrow\N$. We denote the rank of $M$, that is $r(E)$, simply by $r$. 

For a fixed field $\K$ consider a $\K$-homomorphism $\varphi_M$ between polynomial rings:
$$\varphi_M:S_M=\K[y_B:B\in\mathfrak{B}]\ni y_B\rightarrow\prod_{e\in B}x_e\in\K[x_e:e\in E].$$
The \emph{toric ideal of a matroid} $M$, denoted by $I_M$, is the kernel of the map $\varphi_M$. For a representable matroid $M$ the toric variety associated with the toric ideal $I_M$ has a very nice embedding as a subvariety of a Grassmannian 
\cite{GeGoMaSe87}. It is the closure of the torus orbit of the point in the Grassmannian corresponding to the matroid $M$. Furthermore, any closure of the torus orbit of a point in the Grassmannian is of this form for some representable matroid $M$.

When an ideal is defined only by combinatorial means, one expects to have a combinatorial description of its set of generators. An attempt to achieve this description often leads to surprisingly deep combinatorial questions. White's conjecture is an example. In $1980$ Neil White stated in fact three conjectures that describe generators of the toric ideal of a matroid with increasing accuracy.

\begin{conjecture}[White, \cite{Wh80}]\label{ConjectureWeakWhite}
The toric ideal $I_M$ of a matroid $M$ is generated by quadratic binomials.
\end{conjecture}

The family $\mathfrak{B}$ of bases of the matroid $M$ satisfies \emph{symmetric exchange property} (the reader is referred to \cite{Ox92} for background of matroid theory, and to \cite{La15} for other exchange properties). That is, for every bases $B_1,B_2$ and an element $e\in B_1\setminus B_2$ there exists an element $f\in B_2\setminus B_1$, such that both sets $B_1'=(B_1\setminus e)\cup f$ and $B_2'=(B_2\setminus f)\cup e$ are bases. In such a case we say that the quadratic binomial $y_{B_1}y_{B_2}-y_{B_1'}y_{B_2'}$ \emph{corresponds to symmetric exchange}. It is clear that this binomials belong to the ideal $I_M$. 

\begin{conjecture}[White, \cite{Wh80}]\label{ConjectureClassicWhite}
The toric ideal $I_M$ of a matroid $M$ is generated by quadratic binomials corresponding to symmetric exchanges. 
\end{conjecture}

\begin{conjecture}[White, \cite{Wh80}]\label{ConjectureStrongWhite}
The toric ideal $I_M$ of a matroid $M$ considered in the noncommutative polynomial ring $\K\langle y_B:B\in\mathfrak{B}\rangle$ is generated by quadratic binomials corresponding to symmetric exchanges. 
\end{conjecture}

Conjecture \ref{ConjectureStrongWhite}, the strongest among White's conjectures describing generators of the ideal $I_M$, turned out to be equivalent to Conjecture \ref{ConjectureClassicWhite} when considered for all matroids (see the discussion in Section $4$ of \cite{LaMi14}). 

Since every toric ideal is generated by binomials, it is not hard to rephrase the above conjectures in the combinatorial language. Conjecture \ref{ConjectureWeakWhite} asserts that if two multisets of bases of a matroid have equal union (as a multiset), then one can pass between them by a sequence of steps, in each step exchanging two bases for another two bases with the same union (as a multiset). In Conjecture \ref{ConjectureClassicWhite} additionally each step corresponds to a symmetric exchange. In Conjecture \ref{ConjectureStrongWhite} one takes sequences of bases instead of multisets, and similarly each step corresponds to a symmetric exchange between consecutive bases. Actually, these are the original formulations due to White. It is immediate that the conjectures do not depend on the field $\K$.

In $2002$ J\"{u}rgen Herzog and Takayuki Hibi, as they wrote, could not escape from the temptation to ask the following questions strengthening Conjecture \ref{ConjectureWeakWhite}:

\begin{question}[Herzog, Hibi, \cite{HeHi02}]\label{QuestionHerzogHibi}
	Let $M$ be a matroid.
	\begin{enumerate}
		\item Is the base ring $S_M/I_M$ a Koszul algebra?
		\item Does the toric ideal $I_M$ possess a quadratic Gr\"{o}bner basis?
	\end{enumerate}
\end{question}

Recall general implications (see ex. \cite{CoDeRo13,Co14}). If the quotient ring $S/I$ (where $S$ is the polynomial ring) is a Koszul algebra, then the ideal $I$ is generated by quadrics. If the ideal $I$ has a quadratic Gr\"{o}bner basis, then $S/I$ is a Koszul algebra.

\subsection{State of the art}

White's conjectures are known to be true for many special classes of matroids: graphic matroids \cite{Bl08}, linear frame matroids \cite{Mc20}, strongly base orderable matroids \cite{LaMi14} (in particular for transversal matroids), sparse paving matroids \cite{Bo13}, and for matroids of low rank -- at most $3$ \cite{Ka10} (see also other related papers \cite{HeHi02,HiLaMaMiVo19,LaMi20,Sh16}). 

Questions of Herzog and Hibi are settled for a fewer number of matroid classes. Base-sortable matroids (called also sort-closed matroids) possess quadratic Gr\"{o}bner basis \cite{LaPo07,Bl01} (in particular uniform matroids \cite{St95} and lattice path matroids \cite{Sc11} do), as their base polytope is alcoved hence possesses a quadratic regular unimodular triangulation. Moreover, the answer to Question \ref{QuestionHerzogHibi} $(1)$ is `yes' for transversal polymatroids \cite{Co07}.

The first and to our knowledge the only known `general result', i.e. a result valid for arbitrary matroids, confirmed White's Conjecture \ref{ConjectureClassicWhite} `up to saturation'. To explain what this means, denote by $\mathfrak{m}$ the so-called \emph{irrelevant ideal}, i.e. the ideal generated by all variables in the polynomial ring $S_M$. Recall that the ideal $I:\mathfrak{m}^{\infty}=\{a\in S_M:a\mathfrak{m}^n\subset I\text{ for some }n\in\N\}$ is called the \emph{saturation} of the ideal $I$ with respect to $\mathfrak{m}$. Notice that the ideal $I_M$, as a prime ideal, is saturated. Let $J_M$ be the ideal generated by quadratic binomials corresponding to symmetric exchanges. Clearly, $J_M\subset I_M$. Conjecture \ref{ConjectureClassicWhite} asserts that the ideals $J_M$ and $I_M$ are equal, which translates geometrically as `they define the same affine scheme'.

In \cite{LaMi14} we prove that the saturations of $I_M$ and $J_M$ with respect to $\mathfrak{m}$ are equal. That is, in algebraic geometry language, `they define the same projective scheme'. In particular, ideals $I_M$ and $J_M$ have the same affine set of zeros, so Conjecture \ref{ConjectureClassicWhite} holds on `set-theoretic level'. Recall that homogeneous ideals have equal saturations with respect to $\mathfrak{m}$ if and only if their homogeneous parts are equal starting from some degree. Thus, we can rephrase the result of \cite{LaMi14} in the following way:

\begin{theorem}[Laso{\'n}, Micha{\l}ek, \cite{LaMi14}]\label{TheoremLasonMichalek}
For every matroid $M$ there exists a constant $c(M)$ such that the homogeneous parts of degree at least $c(M)$ of the toric ideal $I_M$ are generated by quadratic binomials corresponding to symmetric exchanges.
\end{theorem}

Moreover, from the proof of the theorem follows that $c(M)=d+(r-1)d\lvert\frak{B}\rvert$, where $r$ is the rank of $M$, $\lvert\frak{B}\rvert$ is the size of the set of bases of $M$, and $d$ is the degree in which the ideal $I_M$ is generated.

\subsection{Our results}

This paper studies toric ideals of matroids of a fixed rank. We obtain several finiteness results. Namely, for an infinite class of matroids of a fixed rank we obtain finite bounds for various parameters -- degree of a Gr\"{o}bner basis, complexity of verifying White's conjecture, shape of the Betti table, and degree starting from which White's conjecture holds. 

Firstly, towards the answer to the Question \ref{QuestionHerzogHibi} of Herzog and Hibi, we make an initial step in the study of Gr\"{o}bner bases of the toric ideal of a general matroid:

\begin{customtheorem}{\ref{Theorem2}}
Let $M$ be a matroid of rank $r$. The toric ideal $I_M$ has a Gr\"{o}bner basis of degree at most $(r+3)!$. In particular, it is generated in degree at most $(r+3)!$.  
\end{customtheorem}

As a result, White's conjectures for matroids of a fixed rank are finite problems:

\begin{customcorollary}{\ref{Corollary1}}
Checking if Conjecture \ref{ConjectureWeakWhite}, \ref{ConjectureClassicWhite} or \ref{ConjectureStrongWhite} is true for matroids of a fixed rank is decidable (it reduces to connectivity of a finite number of graphs).
\end{customcorollary}

Moreover, we get a bound for the shape of the nonzero part of the Betti table:

\begin{customcorollary}{\ref{Corollary2}}
Let $M$ be a matroid of rank $r$. The $i$-th graded Betti number of the toric ideal $I_M$ vanishes for degrees greater than $(r+3)!(i+1)$.
\end{customcorollary}

Finally, the main result of the paper strengthens Theorem \ref{TheoremLasonMichalek}:

\begin{customtheorem}{\ref{Theorem1}}
For every $r$ there exists a constant $c(r)$ such that for every matroid $M$ of rank $r$ the homogeneous parts of degree at least $c(r)$ of the toric ideal $I_M$ are generated by quadratic binomials corresponding to symmetric exchanges.
\end{customtheorem}

By the discussion after Remark $15$ in \cite{LaMi14} we can deduce the following:

\begin{remark}
Theorems \ref{Theorem2}, \ref{Theorem1} and Corollaries \ref{Corollary1}, \ref{Corollary2} are true for discrete polymatroids.
\end{remark}


\subsection{Outline of the paper}

In the next section we discuss how White's conjectures translate into problems on graphs on bases of a matroid. We use it later in the proof of Corollary \ref{Corollary1}.

In Section \ref{3} we prove Theorem \ref{TheoremGeneral}, a structural result on bases coverings. This is a main tool for Theorem \ref{Theorem2} which bounds the degree of a Gr\"{o}bner basis of the toric ideal of a matroid only in terms of its rank. Theorem \ref{Theorem2} leads to Corollaries \ref{Corollary1}, \ref{Corollary2}, and later serves as an ingredient for the proof of Theorem \ref{Theorem1}.

The main result of the paper strengthens Theorem \ref{TheoremLasonMichalek}. Recall that the key part in the proof of Theorem \ref{TheoremLasonMichalek} is \cite[Claim $4$]{LaMi14}. It asserts that if $b\in I_M$ is a binomial of degree $n$, then for every variable $y_B$ we have $y_B^{(r-1)n}b\in J_M$. Suppose the ideal $I_M$ is generated in degree $d$. If $b\in I_M$ is a binomial of degree at least $d+(r-1)d\lvert\frak{B}\rvert$, then $b=a_1b_1+\dots+a_kb_k$ where $b_i$ are generators of $I_M$ of degree $d$, and $a_i$ are monomials of degree at least $(r-1)d\lvert\frak{B}\rvert$. From the pigeon hole principle, every monomial $a_i$ contains some variable $y_B$ in degree at least $(r-1)d$, hence by \cite[Claim $4$]{LaMi14} $a_ib_i\in J_M$, and finally $b\in J_M$. Therefore, the constant $c(M)$ from Theorem \ref{TheoremLasonMichalek} can be bounded by $d+(r-1)d\lvert\frak{B}\rvert$. 

By Theorem \ref{Theorem2} we have a bound $d\leq (r+3)!$. But, the size of the set of bases $\lvert\frak{B}\rvert$ can not be bounded for matroids of rank $r$. Indeed, we have to be able to generate by quadratic binomials corresponding to symmetric exchanges binomials $y_{B_1}\cdots y_{B_n}-y_{B'_1}\cdots y_{B'_n}\in I_M$ of high degree with respect to the rank ($n\gg r$) for which bases $B_1,\dots,B_n$ are pairwise disjoint. For them there is no hope for a single variable in high degree, as every variable can appear in degree at most one. 

To overcome this difficulty, in Section \ref{5}, we introduce a Ramsey-type result for blow-ups of bases. It asserts that if a matroid contains sufficiently many disjoint bases, then it contains an arbitrarily large \emph{$k$-th blow-up of a basis} -- a matroid obtained by replacing every element of a basis by $k$ parallel elements. Moreover, if we modify this bases by only symmetric exchanges, then we can guarantee that this $k$-th blow-up agrees with some $k$ bases. This allows us to `reveal' a single variable in high degree in any monomial $y_{B_1}\cdots y_{B_n}$ of sufficiently large degree. 

Having these three ingredients -- \cite[Claim $4$]{LaMi14}, Theorem \ref{Theorem2}, and a Ramsey-type result for blow-ups of bases, we finally prove Theorem \ref{Theorem1} in the last section. 

\section{Graphs on bases of a matroid}\label{2}

This section contains preliminaries, in particular notions used in Section \ref{3}. We discuss here how White's conjectures translate into problems on graphs on bases of a matroid. 

We say that two bases of a matroid are \emph{neighboring} if one is obtained from the other by a symmetric exchange. That is, if their symmetric difference has exactly two elements. A graph on bases of a matroid $M$ with edges between neighboring bases is called the \emph{basis graph} of $M$, and denoted by $\mathfrak{B}(M)$. Basis graphs have been studied in $1960$s and $1970$s, and they are well understood. In particular, basis graphs are Hamiltonian (with only two trivial exceptions), even a characterization is known (see \cite{Ma73,Ma74,ChChOs15} and references within).

For $k\geq 1$, a \emph{$k$-matroid} is a matroid whose ground set can be partitioned into $k$ pairwise disjoint bases. We call a basis of a $k$-matroid \emph{complementary} if its complement can be partitioned into $k-1$ pairwise disjoint bases. That is, when it is an element of some partition of the ground set into bases. When $\mathfrak{B}$ is the set of bases of a $k$-matroid, then we denote the set of complementary bases by $\mathfrak{B}^c$.

We recall one of the versions of the matroid union theorem, which will be used several times in this paper. It characterizes $k$-matroids in terms of rank function.

\begin{theorem}[Nash-Williams \cite{Na64}, Edmonds \cite{Ed65}]\label{TheoremMatroidUnion}
A matroid $M$ is a $k$-matroid, if and only if for every $A\subset E$ the inequality $k\cdot r(A)\geq\lvert A\rvert$ holds, and $k\cdot r(E)=\lvert E\rvert$.
\end{theorem}

Blasiak \cite{Bl08} proposed a very nice and simple translation of the problem of generating the toric ideal of a matroid to the problem of connectivity of some graphs naturally associated to $k$-matroids. We are going to use this approach for the proof of Corollary \ref{Corollary1}. Following Blasiak, for $k\geq 3$ the \emph{$k$-base graph} of a $k$-matroid $M$, denoted by $\mathfrak{B}_k(M)$, is a graph on sets of $k$ pairwise disjoint bases of $M$ (partitions of the ground set into bases), where edges join vertices with nonempty intersection. That is, sets of bases $\{B_1,\dots,B_k\}$ and $\{B_1',\dots,B_k'\}$ are connected in $\mathfrak{B}_k(M)$ if for some $i,j$ the equality $B_i=B_j'$ holds. Recall that if $\{e,f\}$ is a circuit in a matroid $M$, then elements $e$ and $f$ are said to be \emph{parallel}. In this case $B$ is a basis of $M$ containing $e$ if and only if $(B\cup f)\setminus e$ is a basis of $M$ containing $f$. Via this property one can add to a matroid elements parallel to a fixed element (enlarging its ground set), or remove them. Notice that the reflexive closure of being parallel is an equivalence relation. A simple corollary of the proof of Proposition $2.1$ from \cite{Bl08} gives the following.

\begin{proposition}[Blasiak, \cite{Bl08}]\label{PropositionBlasiakReduction}
Let $\mathfrak{C}$ be a class of matroids that is closed under deletions that do not lower the rank of a matroid, and adding parallel elements. Then the following conditions are equivalent:
\begin{enumerate}
\item for every $k>d$ and for every $k$-matroid $M$ in $\mathfrak{C}$ the $k$-base graph $\mathfrak{B}_k(M)$ is connected, 
\item for every matroid $M$ in $\mathfrak{C}$ the ideal $I_M$ is generated in degree at most $d$.
\end{enumerate}
\end{proposition}

In particular, in order to prove Conjecture \ref{ConjectureWeakWhite} it is enough to show that for every $k>2$ and for every $k$-matroid $M$ the $k$-base graph $\mathfrak{B}_k(M)$ is connected.

\smallskip

Here we propose another approach to White's conjecture. Consider other graphs that can be naturally associated to $k$-matroids. The \emph{complementary basis graph} of a $k$-matroid, denoted by $\mathfrak{B}^c(M)$, is a graph on complementary bases of $M$ with edges between neighboring bases. That is, the complementary basis graph of a $k$-matroid is the restriction of its basis graph to complementary bases $\mathfrak{B}^c(M)=\mathfrak{B}(M)\vert_{\mathfrak{B}^c}$.

Graphs $\mathfrak{B}^c(M)$ have been already studied for $2$-matroids $M$. In 1985 Farber, Richter and Shank \cite{FaRiSh85} proved that for a graphic $2$-matroid $M$ the graph $\mathfrak{B}^c(M)$ is connected, they also conjectured connectivity for arbitrary $2$-matroids. In \cite{Bl08} after the proof of Proposition $2.1$ Blasiak observes the following easy equivalence:

\begin{proposition}[Blasiak, \cite{Bl08}]\label{PropositionBlasiakReductionDeg2}
Let $\mathfrak{C}$ be a class of matroids that is closed under deletions that do not lower the rank of a matroid, and adding parallel elements. Then the following conditions are equivalent:
\begin{enumerate}
\item for every $2$-matroid $M$ in $\mathfrak{C}$ the complementary basis graph $\mathfrak{B^c}(M)$ is connected,
\item for every matroid $M$ in $\mathfrak{C}$, elements of degree $2$ in $I_M$ considered in the noncommutative polynomial ring $\K\langle y_B:B\in\mathfrak{B}\rangle$ are generated by quadratic binomials corresponding to symmetric exchanges.
\end{enumerate}
\end{proposition}

We state the following two conjectures strongly related to White's conjectures:

\begin{conjecture}\label{ConjectureConnected}
Complementary basis graph of a $k$-matroid is connected.
\end{conjecture}

\begin{conjecture}\label{Conjecturekr+1}
Let $k\geq 2$, and let $M$ be a matroid of rank $r$ on the ground set $E$ of size $kr+1$. Suppose $x,y\in E$ are two elements such that both sets $E\setminus x$ and $E\setminus y$ can be partitioned into $k$ pairwise disjoint bases. Then there exist partitions of $E\setminus x$ and $E\setminus y$ into $k$ pairwise disjoint bases which share a common basis.
\end{conjecture}

We learned from Joseph Bonin that Conjecture \ref{Conjecturekr+1} for $k=2$ was studied in $1980$s by Paul Seymour and Neil White, but it was not resolved. 

\begin{proposition}
Let $\mathfrak{C}$ be a class of matroids that is closed under deletions that do not lower the rank of a matroid, and adding parallel elements. Then, considered for all matroids in $\mathfrak{C}$, the following implications between conjectures hold:
\begin{enumerate}
\item the strongest White's Conjecture \ref{ConjectureStrongWhite} implies complementary basis graph Conjecture \ref{ConjectureConnected},
\item conjunction of complementary basis graph Conjecture \ref{ConjectureConnected} and Conjecture \ref{Conjecturekr+1} implies the strongest White's Conjecture \ref{ConjectureStrongWhite},\newline 
in general, for every $d\geq 2$ conjunction of Conjecture \ref{ConjectureConnected} for $k$-matroids for $k>d$ and Conjecture \ref{Conjecturekr+1} for $k\geq d$ implies that the toric ideal of a matroid, considered in the noncommutative polynomial ring, is generated in degree at most $d$.
\end{enumerate}
\end{proposition}

\begin{proof}
We begin with implication $(1)$. Let $M$ be a $k$-matroid in $\mathfrak{C}$, and let $B_1,B'_1$ be complementary bases in $M$. So, there exist bases $B_2,\dots,B_k$, $B'_2,\dots,B'_k$ such that entries of the sequences $\mathcal{A}=(B_1,\dots,B_k)$ and $\mathcal{A}'=(B'_1,\dots,B'_k)$ form partitions of the ground set $E$. Then $b=y_{B_1}\cdots y_{B_k}-y_{B'_1}\cdots y_{B'_k}\in I_M$, or equivalently sequences of bases $\mathcal{A}$ and $\mathcal{A}'$ have equal union (as a multiset). By the assumption, we can generate $b$ using quadratic binomials corresponding to symmetric exchanges, or equivalently we can pass between $\mathcal{A}$ and $\mathcal{A}'$ by a sequence of steps, in each step making a symmetric exchange. Notice that all bases appearing during this process are complementary bases of $M$. Observe that the first bases of two sequences joined by a single step are either the same or neighboring. Thus we get a path in $\mathfrak{B}^c(M)$ between $B_1$ and $B'_1$.

For the implication $(2)$, by Propositions \ref{PropositionBlasiakReduction} and \ref{PropositionBlasiakReductionDeg2} it is enough to show that for every $k\geq 3$ and for every $k$-matroid $M$ in $\mathfrak{C}$ the $k$-base graph $\mathfrak{B}_k(M)$ is connected, and for every $2$-matroid $M$ in $\mathfrak{C}$ the complementary basis graph $\mathfrak{B^c}(M)$ is connected. The second part we get directly from complementary basis graph Conjecture \ref{ConjectureConnected}. For the first part, let $M$ be a $k$-matroid in $\mathfrak{C}$ (for $k\geq 3$) and let $\{B_1,\dots,B_k\}$, $\{B_1',\dots,B_k'\}$ be two vertices in $\mathfrak{B}_k(M)$. Since $\mathfrak{B^c}(M)$ is connected, it is enough to show that vertices of $\mathfrak{B}_k(M)$ containing neighboring bases are connected in $\mathfrak{B}_k(M)$. If the symmetric difference of $B_1$ and $B'_1$ is $\{x,y\}$, then consider the restriction of $M$ to the set $E\setminus (B_1\cap B'_1)$. This matroid satisfies assumptions of Conjecture \ref{Conjecturekr+1} with points $x,y$. Thus in $\mathfrak{B}_k(M)$ there are vertices $\{B_1,B''_2,\dots,B''_k\},\{B_1',B'''_2,\dots,B_k'''\}$ connected by an edge. The first one is connected by an edge with $\{B_1,\dots,B_k\}$, while the second with $\{B_1',\dots,B_k'\}$. We get that any two vertices $\{B_1,\dots,B_k\}$, $\{B_1',\dots,B_k'\}$ in $\mathfrak{B}_k(M)$ containing neighboring bases are connected by a path.
\end{proof}

\begin{proposition}
If $k\geq 2^{r-1}+1$, then Conjecture \ref{Conjecturekr+1} holds.
\end{proposition}

\begin{proof}
Proof by contradiction. Let $B_1,\dots,B_k$ be a partition of the set $E\setminus y$ into $k$ pairwise disjoint bases. Without loss of generality $x\in B_1$. If the assertion is not true, then for each $i=2,\dots,k$ the basis $B_i$ can not be completed to a partition of $E\setminus x$ into $k$ bases. Thus from the matroid union Theorem \ref{TheoremMatroidUnion} it follows that for each $i=2,\dots,k$ there is a set $A_i\subset E\setminus (x\cup B_i)$ such that $(k-1)r(A_i)<\lvert A_i\rvert$. On the other hand, since $E\setminus (y\cup B_i)$ has a  partition into $k-1$ pairwise disjoint bases (namely $B_1,\dots,\hat{B_i},\dots,B_k$), for every set $A\subset E\setminus (y\cup B_i)$ the inequality $(k-1)r(A)\geq\lvert A\rvert$ holds. Thus for each $i$ we have $y\in A_i$, $r(A_i)=r(A_i\setminus y)$, and $(k-1)r(A_i\setminus y)=\lvert A_i\setminus y\rvert$. The last equality implies that for every basis $B_j$ (for $j\neq i$) $\lvert B_j\cap A_i\rvert=r(A_i)$. Moreover, since there is equality in the inequality $(k-1)r(A_i\setminus y)\geq\lvert A_i\setminus y\rvert$, each $A_i$ is closed in $E\setminus B_i$, and it is equal to the closure of $B_j\cap A_i$ (for $j\neq i$) in $E\setminus B_i$. Consider sets $B_1\cap A_i\subset B_1\setminus x$. None of them is empty, since otherwise $r(A_i)=0$ and $y$ would be a loop. Thus, since there are $k-1\geq 2^{r-1}$ of them, for some $i\neq j$ the equality $B_1\cap A_i=B_1\cap A_j$ holds. But since $A_i$ is the closure of  $B_1\cap A_i$ in $E\setminus B_i$ and $A_j$ is the closure of $B_1\cap A_j$ in $E\setminus B_j$, we get that the set $A:=A_i\cup A_j$ is the closure of $B_1\cap A_i=B_1\cap A_j$ in $E$, so it is closed, $y\in A$, and $\lvert B_l\cap A\rvert=r(A)$ for every $l=1,\dots,k$. Therefore, $\lvert A\rvert=k\cdot r(A)+1$ and $x\notin A$, which by the matroid union Theorem \ref{TheoremMatroidUnion} contradicts the assumption that $E\setminus x$ can be partitioned into $k$ pairwise disjoint bases.
\end{proof}

\section{Degree bounds and Gr\"{o}bner bases}\label{3}

By Hilbert's basis theorem the ideal $I_M$ is finitely generated. However, it is not easy to give any explicit bound on degree in which it is generated. A bound follows from a more general theorem about toric ideals. Theorem $13.14$ from \cite{St95} asserts that if a graded set  $\mathcal{A}\subset\Z^d$ generates a normal semigroup, then the corresponding toric ideal $I_{\mathcal{A}}$ is generated in degree at most $d$. For a matroid $M$ we consider the set of indicator vectors of all bases, that is $\{e_B=\sum_{i\in B}e_i:B\in\mathfrak{B}\}\subset\Z^{\lvert E\rvert}$, and denote it also by $\mathfrak{B}$. By \cite[Theorem $1$]{Wh77} the semigroup $\N\mathfrak{B}$ generated by $\mathfrak{B}$ is normal (it is also an easy consequence of the matroid union theorem -- Theorem \ref{TheoremMatroidUnion}). The toric ideal corresponding to $\mathfrak{B}$ is the ideal $I_M$. Hence, the toric ideal of a matroid is generated in degree at most the size of its ground set. If we fix the size of the ground set, then there are only finitely many matroids on it, so a common bound is not surprising. 

When we fix only the rank, then the number of matroids of that rank is infinite. Theorem  \ref{Theorem2} asserts that in this case there is also a common bound on the degree in which the corresponding toric ideals are generated. Furthermore, there is a common bound on the degree of their Gr\"{o}bner bases. In order to prove it we need the following structural result:

\begin{theorem}\label{TheoremGeneral}
Let $M$ be a $k$-matroid of rank $r$. Among every $k-s$ pairwise disjoint bases in $M$ at most $r(r+2)!+s(r+1)!$ bases are not complementary.
\end{theorem}

\begin{proof} 
We prove by induction on $r$, that among every $k-s$ pairwise disjoint bases there are at most $r(r+2)!+s((r+1)!-2)$ not complementary bases. When $r=1$, then the statement becomes trivial. Suppose $r\geq 2$, and fix $k,s$. Let $B_{1},\dots,B_{k}$ be disjoint bases of the $k$-matroid $M$. Their union is the whole ground set $E$. Let $D_{1},\dots,D_{k-s}$ be arbitrary pairwise disjoint bases in $M$. If every $D_j$ is complementary, then the assertion clearly holds. So, we can assume that some basis $D_{j}$ is not complementary.

Due to the matroid union theorem -- Theorem \ref{TheoremMatroidUnion}, there exists a set $A\subset E\setminus D_{j}$, such that
$$(k-1)\cdot r(A)<\lvert A\rvert.$$
Of course $0<r(A)<r$. Indeed, otherwise either $A$ would have to be empty (in a $k$-matroid $M$ there are no loops) and we would have $0<0$, or we would have $(k-1)\cdot r<\lvert A\rvert\leq\lvert E\setminus D_{j}\rvert=(k-1)\cdot r$.

Let $A_i=A\cap B_i$ for every $i=1,\dots,k$. Since every $B_i$ is a basis, inequalities $\lvert A_i\rvert\leq r(A)$ hold. And, all together
$$(k-1)\cdot r(A)<\lvert A\rvert=\lvert A_1\rvert+\dots+\lvert A_k\rvert\leq k\cdot r(A).$$
Therefore for every $i$, except at most $r(A)-1\leq r$, we have $\lvert A_i\rvert=r(A)$. Without loss of generality the equality holds for $i=1,\dots,k-r$. 

Let $E'=B_{1}\cup\dots\cup B_{k-r}$ and let $A'=A_{1}\cup\dots\cup A_{k-r}$. We are going to reduce the problem to the $(k-r)$-matroid $M'=M\vert_{E'}$ (restriction of $M$ to the set $E'$), and then use the set $A'$ to split it into smaller instances -- for $M'\vert_{A'}$ (restriction of $M'$ to $A'$) and for $M'/A'$ (contraction of $A'$ in $M'$).

Notice that there are at most $r^2$ bases among bases $D_i$ which have non-empty intersection with $B_{k-r+1}\cup\dots\cup B_k$. Thus, without loss of generality, bases $D_1,\dots,D_{k-s-r^2}$ are contained in $E'$.

Let $C_i=A'\cap D_i$ for $i=1,\dots,k-s-r^2$. Since every $D_i$ is a basis, inequalities $\lvert C_i\rvert\leq r(A')=r(A)$ hold. In order to split the problem for $M'\vert_{A'}$ and $M'/A'$ we need bases $D_i$ satisfying $\lvert C_i\rvert=\lvert A'\cap D_i\rvert=r(A)$. 

Since $D_{1},\dots,D_{k-s-r^2}$ cover all except $(s+r^2)r$ elements of $E$ we get
$$(k-s-r^2)r(A)-(s+r^2)r\leq (k-r)r(A)-(s+r^2)r=\lvert A'\rvert-(s+r^2)r\leq$$ 
$$\leq\lvert A'\cap(D_1\cup\dots\cup D_{k-s-r^2})\rvert=\lvert C_1\rvert+\dots+\lvert C_{k-s-r^2}\rvert\leq (k-s-r^2)r(A).$$
Therefore for every $i=1,\dots,k-s-r^2$, except at most $(s+r^2)r$, the equality $\lvert C_i\rvert=r(A)$ holds. Without loss of generality it holds for $i=1,\dots,k-(s+r^2)(r+1)$. Denote $s'=(s+r^2)(r+1)-r$. Now we can pass to the matroids $M'\vert_{A'}$ and $M'/A'$. 

We have
\begin{enumerate}
\item $(k-r)$-matroid $M'\vert_{A'}$ of rank $r(A)<r$ with $(k-r)-s'$ pairwise disjoint bases $D_1\cap A',\dots,D_{k-r-s'}\cap A'$, and
\item $(k-r)$-matroid $M'/A'$ of rank $r-r(A)<r$ with $(k-r)-s'$ pairwise disjoint bases $D_1\setminus A',\dots,D_{k-r-s'}\setminus A'$. 
\end{enumerate}

For both cases we use the inductive assumption. In the case $(1)$ there are at most $r(A)(r(A)+2)!+s'(r(A)+1)!-2s'$ bases not complementary in $M'\vert_{A'}$. In the case $(2)$ there are at most $(r-r(A))(r-r(A)+2)!+s'(r-r(A)+1)!-2s'$ bases not complementary in $M'/A'$. 

Notice that if $D_i\cap A'$ is a basis complementary in $M'\vert_{A'}$ and $D_i\setminus A'$ is a basis complementary in $M'/A'$, then $D_i$ is a basis complementary in $M$. Therefore all together there are at most 
$$t=(s'+r)+r(A)(r(A)+2)!+s'(r(A)+1)!-2s'$$
$$+(r-r(A))(r-r(A)+2)!+s'(r-r(A)+1)!-2s'$$
bases not complementary in $M$. Denote $r'=r(A)$ and recall that $0<r'<r$. Now,
$$t=r+r'(r'+2)!+(r-r')(r-r'+2)!+(r^2(r+1)-r)(1+(r'+1)!-2+(r-r'+1)!-2)+$$
$$+s(r+1)(1+(r'+1)!-2+(r-r'+1)!-2)\leq r(r+2)!+s((r+1)!-2).$$
It is easy to verify the last inequality, because for $s\geq 0$ and $r\geq 2$ we have both
$$s(r+1)(1+r!-2+2!-2)\leq s((r+1)!-2),$$
$r+(r-1)(r+1)!+3!+(r^2(r+1)-r)(1+r!-2+2!-2)\leq r(r+2)!.$
\end{proof}

We do not have faith in a positive answer to Question \ref{QuestionHerzogHibi} $(2)$ of Herzog and Hibi -- that is, existence of a quadratic Gr\"{o}bner basis for the toric ideal of a matroid. That would mean (see \cite{St95}) that the base polytope of a matroid possesses a quadratic regular unimodular triangulation. This property is nearly at the top of the whole hierarchy of possible combinatorial properties of a convex lattice polytope (consult \cite{HaPaPiSa14}). It is known that the base polytope of a matroid has the integer Carath\'eodory property \cite{GiRe12}, and therefore it is normal (which was earlier shown in \cite[Theorem $1$]{Wh77}). However, using Theorem \ref{TheoremGeneral} we are able to prove a bound on the degree of a Gr\"{o}bner basis that depends only on the rank of a matroid:

\begin{theorem}\label{Theorem2}
Let $M$ be a matroid of rank $r$. The toric ideal $I_M$ has a Gr\"{o}bner basis of degree at most $(r+3)!$. In particular, it is generated in degree at most $(r+3)!$.  
\end{theorem}

\begin{proof}
Let $<$ be any fixed order on variables $y_B$, and let $<_{revlex}$ be the corresponding reverse lexicographic order on monomials in the ring $\K[y_B:B\in\mathfrak{B}]$. We prove by induction on $n$, that if $b\in I_M$ is a binomial of degree $n$, then its leading monomial is divisible by the leading monomial of some binomial in $I_M$ of degree at most $(r+3)!$. This implies that the set of monic binomials of degree at most $(r+3)!$ forms a Gr\"{o}bner basis of $I_M$ with respect to $<_{revlex}$.

If $n\leq (r+3)!$, then it is clear. Suppose $n>(r+3)!$ and that the assertion holds for all degrees less than $n$. Let $b=y_{B_1}\cdots y_{B_n}-y_{D_1}\cdots y_{D_n}\in I_M$ be a binomial, where the first monomial is leading. Suppose, without loss of generality, that $y_{D_1}$ is a minimal variable among variables $y_{D_i}$. Notice that if for some $i,j$ equality $B_i=D_j$ holds, then $b=y_{B_i}b'$ for some binomial $b'\in I_M$ and the assertion follows from the inductive assumption for $b'$ which is of degree $n-1$. Therefore, since $y_{D_1}\cdots y_{D_n}<_{revlex}y_{B_1}\cdots y_{B_n}$, we have $y_{D_1}<y_{B_i}$ for every $i$. 

Let $M'$ be a $n$-matroid obtained from the matroid $M\vert_{B_1\cup\dots\cup B_n}$ by replacing every element $e\in E$ by $w_e=\lvert\{i:e\in B_i\}\rvert$ parallel elements $(e,1),\dots,(e,w_e)$. The ground set $E'$ of $M'$ has two natural partitions into $n$ bases, namely $B'_1\cup\dots\cup B'_n$ and $D'_1\cup\dots\cup D'_n$, where $B'_i=\{(e,v_e):e\in B_i,v_e=\lvert\{j\leq i:e\in B_j\}\rvert\}$ and $D'_i$ are defined analogously. Denote by $T$ the set of indices $i$ such that $B'_i$ intersects basis $D'_1$, clearly $t=\lvert T\rvert\leq r$. Consider $(n-1)$-matroid $M'\vert_{D'_2\cup\dots\cup D'_n}$. By Theorem \ref{TheoremGeneral} among $n-t\geq n-r$ disjoint bases $B'_i$ for $i\in\{1,\dots,n\}\setminus T$ there are at most $r(r+2)!+(r-1)(r+1)!<(r+3)!-r<n-r$ not complementary bases. Thus, there exists a basis $B'_k$ that can be completed by some bases $H'_1,\dots,H'_{n-2}$ to $E'\setminus D'_1$. Let $H_i=\{e\in E:\exists_j(e,j)\in H'_i\}$ be the basis in $M$ corresponding to $H'_i$. Now $b'=y_{B_1}\cdots y_{B_n}-y_{D_1}y_{B_k}y_{H_1}\cdots y_{H_{n-2}}\in I_M$ is a binomial in which $y_{B_1}\cdots y_{B_n}$ is the leading monomial. Moreover, $b=y_{B_k}b'$ for some binomial $b'\in I_M$ and the assertion follows by induction.
\end{proof}

\begin{corollary}\label{Corollary1}
Checking if Conjecture \ref{ConjectureWeakWhite}, \ref{ConjectureClassicWhite} or \ref{ConjectureStrongWhite} is true for matroids of a fixed rank is decidable (it reduces to connectivity of a finite number of graphs).
\end{corollary}

\begin{proof}
Using Proposition \ref{PropositionBlasiakReduction}, in order to check if Conjecture \ref{ConjectureWeakWhite} is true for matroids of rank $r$ it is enough to check if for every $k$-matroid $M$ of rank $r$ (for every $k>2$), the $k$-base graph $\mathfrak{B}_k(M)$ is connected. By Theorem \ref{Theorem2} it is enough to consider $k$ from the range $(r+3)!\geq k>2$, since for $k>(r+3)!$ the statement is true. That is, the problem reduces to checking connectivity of a finite number of graphs. 

To check if Conjectures \ref{ConjectureWeakWhite} and \ref{ConjectureStrongWhite} are equivalent for matroids of rank $r$, by Proposition \ref{PropositionBlasiakReductionDeg2} it suffices to check connectivity of a finite number of graphs.

Analogously, to check if Conjectures \ref{ConjectureWeakWhite} and \ref{ConjectureClassicWhite} are equivalent for matroids of rank $r$ it suffices to check connectivity of graphs from Proposition \ref{PropositionBlasiakReductionDeg2} modified by adding an edge between every complementary basis $B$ and its complement $B^c$ (these are bases of a $2$-matroid). This completes the proof of Corollary \ref{Corollary1}.
\end{proof}

We get a new class of discrete polymatroids for which White's Conjecture \ref{ConjectureWeakWhite} is true (for an extension of White's conjectures to discrete polymatroids see \cite{HeHi02}):

\begin{corollary}\label{Corollary3}
Let $P$ be a discrete polymatroid which is a join of $c\geq\frac{1}{2}(r+3)!$ copies of a matroid $M$ of rank $r$ (a basis of $P$ is a union, as a multiset, of $c$ bases of $M$). Then the toric ideal $I_P$ is generated in degree $2$.
\end{corollary}

\begin{proof}
We will prove the following claim. Let $P$ be a discrete polymatroid which is a join of $c$ copies of a matroid $M$. Suppose that the toric ideal $I_M$ is generated in degree at most $2c$. Then the toric ideal $I_P$ is generated in degree $2$.

Let $D_1,\dots,D_k,\tilde{D}_1,\dots,\tilde{D}_k$ be bases of $P$ with $y_{D_1}\cdots y_{D_k}-y_{\tilde{D}_1}\cdots y_{\tilde{D}_k}\in I_P$. For $i=1,\dots,k$ and $j=1,\dots,c$ let $B_i^j$ and $\tilde{B}_i^{j}$ be bases of $M$ such that $D_i=\bigcup_j B_i^j$ and $\tilde{D}_i=\bigcup_j \tilde{B}_i^{j}$. Then $\prod_{i,j}y_{B_i^j}-\prod_{i,j}y_{\tilde{B}_i^{j}}\in I_M$. 

When one exchanges bases $B_i^j$ and $B_{i'}^{j'}$ between bases $D_i$ and $D_{i'}$ of $P$, then the corresponding elements of $I_P$ differ by an element generated in degree $2$. Thus we can rearrange bases $B_i^j$ (and $\tilde{B}_i^{j}$) into an arbitrary $k$ multisets of $c$ bases. Since $I_M$ is generated in degree $2c$, one can pass between the multisets of bases $\{B_i^j:i,j\}$ and $\{\tilde{B}_i^{j}:i,j\}$ by a sequence of steps, in each step exchanging $2c$ bases for another $2c$ bases of the same union (as a multiset). We partition these $2c$ bases into an arbitrary $2$ parts of $c$ bases. Each part corresponds to a basis of $P$. This way we are able to pass between the multisets of bases $\{D_i\}_i$ and $\{\tilde{D}_i\}_i$ of $P$ by a sequence of steps, in each step exchanging only $2$ bases and preserving multiset union.
\end{proof}

White's conjectures focus on generators of the toric ideal of a matroid. We believe the study of properties of its minimal free resolution may be very significant. The toric ideal of a matroid is graded, so it has a graded free resolution, invariants of which are summarized in the table of graded Betti numbers. The Betti table plays a central role in modern commutative algebra, as it provides a lot of algebraic information about the ideal, indicating e.g. regularity, projective dimension, the Cohen-Macaulay property, the Gorenstein property (for the toric ideal of a matroid see \cite{HiLaMaMiVo19,LaMi20}), the Koszul property, and many others. Astonishingly little is known about the structure of Betti tables associated to arbitrary matroids. One of the possible explanations is the fact that providing the first step of the minimal free resolution is already as hard as White's conjecture -- as the zeroth graded Betti number of degree $d$ is equal to the minimum number of generators of degree $d$. 

We aim at a description of the shape of the nonzero part of the Betti tables of the toric ideals of matroids of a fixed rank:

\begin{corollary}\label{Corollary2}
	Let $M$ be a matroid of rank $r$. The $i$-th graded Betti number of the toric ideal $I_M$ vanishes for degrees greater than $(r+3)!(i+1)$.
\end{corollary}

\begin{proof}
	By Theorem \ref{Theorem2} the toric ideal $I_M$ has an initial ideal $in_{<}(I_M)$ generated in degree at most $(r+3)!$. By the Taylor's resolution of monomial ideals, the $i$-th Betti number of degree $j$ of the monomial ideal $in_{<}(I_M)$ vanishes for $j>(r+3)!(i+1)$. Now, by the standard deformation argument, the $i$-th Betti number of degree $j$ of the original ideal $I_M$ also vanishes for $j>(r+3)!(i+1)$.
\end{proof}

However, we cannot bound the whole nonzero part of the Betti tables of the toric ideals of matroids of a fixed rank: 

\begin{remark}
The projective dimension of the toric ideal $I_M$, that is the index of the largest nonzero Betti number, is not bounded with respect to the rank of $M$.
\end{remark}

\section{Ramsey-type results for blow-ups of bases}\label{5}

By \emph{$k$-th blow-up of a matroid $M$} we mean a matroid obtained from $M$ by replacing every element of its ground set $E$ by $k$ parallel elements. By \emph{$k$-th blow-up of a set $A\subset E$ in $M$} we mean a matroid obtained from $M$ by replacing every element of $A$ by $k$ parallel elements. 

Let $N,M$ be two matroids of the same rank $r$ on the same ground set $E$. We say that $N$ is a \emph{submatroid} of $M$, if the complex of independent sets of $N$ is a subcomplex of the complex of independent sets of $M$, or equivalently, if the set of bases of $N$ is contained in the set of bases of $M$. 

We define a convenient notion of morphisms between matroids. Let $M$ and $M'$ be two matroids of the same rank $r$ on the corresponding ground sets $E,E'$. A \emph{morphism} $\psi$ from $M$ to $M'$ is a function $\psi:E\rightarrow E'$, such that if $B'$ is a basis in $M'$, then any choice of representatives of sets $\{\psi^{-1}(b')\}_{b'\in B'}$ forms a basis in $M$. That is, a function $\psi$ is a morphism if $M$ contains a compatible submatroid which is obtained from $M'$ by replacing every element $e\in E'$ by $\lvert\psi^{-1}(e)\rvert$ parallel copies of $e$. In particular, there is a natural morphism from the $k$-th blow up of a matroid to the original matroid. Let us formulate the key observation:

\begin{observation}\label{ObservationBlowDown}
Suppose $\psi:M\rightarrow M'$ is a morphism between matroids that sends variables of a binomial $b\in I_M$ to variables (i.e. images of the corresponding bases are bases). Let $\psi(b)$ be the binomial $b$ with variables replaced by their images. Then $\psi(b)\in J_{M'}$ implies $b\in J_M$. 
\end{observation}

\begin{proof}
Suppose $b=y_{B_1}\cdots y_{B_k}-y_{D_1}\cdots y_{D_k}$. The condition $\psi(b)\in J_{M'}$ means that one can modify the monomial $y_{\psi(B_1)}\cdots y_{\psi(B_k)}$ using quadratic binomials corresponding to symmetric exchanges in $M'$ to get $y_{\psi(D_1)}\cdots y_{\psi(D_k)}$. But, every symmetric exchange in $M'$ between bases $\psi(B_1),\psi(B_2)$ lifts to a symmetric exchange in $M$ between $B_1,B_2$. Therefore, we can modify $y_{B_1}\cdots y_{B_k}$ using symmetric exchanges in $M$ to get $y_{B'_1}\cdots y_{B'_k}$ such that $\psi(B'_i)=\psi(D_i)$. Since any choice of representatives of bases from $M'$ forms bases in $M$, and $b\in I_M$, we can modify $y_{B'_1}\cdots y_{B'_k}$, step by step using symmetric exchanges, to get $y_{D_1}\cdots y_{D_k}$.
\end{proof}

We say that $k$ disjoint \emph{bases $B_1,\dots,B_k$} of a matroid $M$ \emph{contain the $k$-th blow-up of the basis $B_1$} if there exists a morphism $\psi:M\vert_{B_{1}\cup\dots\cup B_{k}}\rightarrow M\vert_{B_{1}}$, such that $\psi(B_1)=\dots=\psi(B_k)=B_1$ and $\psi$ is an identity on $B_1$. That is, if one can label the elements of $B_{1}\cup\dots\cup B_{k}$ with labels $l_1,\dots,l_r$, elements of every basis $B_i$ with distinct labels, such that every set of $r$ elements of distinct labels is a basis in $M$. 

Our Ramsey-type result asserts that if a matroid contains sufficiently many disjoint bases, then it contains an arbitrarily large $k$-th blow-up of a basis. Moreover, if we modify these bases by only symmetric exchanges, we can guarantee that this $k$-th blow-up agrees with some $k$ bases:

\begin{lemma}\label{LemmaBlowUp}
For every positive integers $r,k$ there exists an integer $n=n(r,k)$, such that if $M$ is an $n$-matroid of rank $r$ with disjoint bases $B_1,\dots,B_n$, then there exists a modification of these bases by symmetric exchanges to $B'_1,\dots,B'_n$ from which one can pick $k$ bases $B_{i_1},\dots,B_{i_k}$ that contain the $k$-th blow-up of basis $B_{i_1}$.
\end{lemma}

\begin{proof}
The proof goes by induction on the rank $r$. If $r=1$, then $M$ itself is the $n$-th blow-up of a basis. In particular, we do not need to make a modification and any $k$ bases contain the $k$-th blow-up of a basis. Thus $n(1,k)=k$.

Suppose $r\geq 2$, and fix also a positive integer $k$. We will show that for $$n=n(r,k)=rn(r-1,k)+r^{rn(r-1,k)}2^{2^{rn(r-1,k)}}n(r-1,k)$$ 
the desired property holds. Denote $s=rn(r-1,k)$, and $t=n-s$. 

Let $M$ be an $n$-matroid of rank $r$ with disjoint bases $B_1,\dots,B_n$. Choose an element $b_i$ in each basis $B_i$ among $B_1,\dots,B_s$. Consider symmetric exchanges between bases $B_1,\dots,B_s$ and bases $B_{s+1},\dots,B_{s+t}$. For $j=s+1,\dots,s+t$ and $i=1,\dots,s$ let $b_{j,i}$ be an element of $B_j$  that exchanges symmetrically with $b_i\in B_i$.

Label elements of each $B_j$ among $B_{s+1},\dots,B_{s+t}$ with distinct labels $l_1,\dots,l_r$. Then each $B_j$ gets a label from $\{l_1,\dots,l_r\}^s$ which is a sequence of labels of $b_{j,1},\dots,b_{j,s}$. From the pigeon hole principle at least $2^{2^{rn(r-1,k)}}n(r-1,k)$ bases $B_j$ have the same label. Without loss of generality (statement of the lemma is independent on the order of bases) they are $B_{s+1},\dots,B_{s+t'}$, for $t'=2^{2^{rn(r-1,k)}}n(r-1,k)$. Now for $j=s+1,\dots,s+t'$ and $i=1,\dots,s$ the label of $b_{j,i}$ is the same for all $j$'s. So, we label $B_i$ with it. Again, from the pigeon hole principle at least $n(r-1,k)$ bases $B_i$ have the same label, without of loss of generality they are $B_1,\dots,B_{n(r-1,k)}$. The label of $b_{j,i}$ is the same for all these $i$'s. Thus, we can define $b_j:=b_{j,i}$. Now, for every $i=1,\dots,n(r-1,k)$ and every $j=s+1,\dots,s+t'$, $b_i\in B_i$ exchanges symmetrically with $b_j\in B_j$.  

Consider matroids $M_j:=(M/b_j)\vert_{B_1\cup\cdots\cup B_{n(r-1,k)}}$ for $j=s+1,\dots,s+t'$. Since there are at most $2^{2^{rn(r-1,k)}}$ matroids on the ground set $B_1\cup\cdots\cup B_{n(r-1,k)}$ (of size $rn(r-1,k)$), there are at least $n(r-1,k)$ indices $j$ for which $M_j$ is the same, without loss of generality for $j=s+1,\dots,s+n(r-1,k)$.

Now, we make a first modification of bases $B_1,\dots,B_n$. We exchange $b_i\in B_i$ symmetrically with $b_{s+i}\in B_{s+i}$ for every $i=1,\dots,n(r-1,k)$, obtaining bases $D_1,\dots,D_n$. Each $D_i$ for $i=1,\dots,n(r-1,k)$ has a distinguished element $d_i$ (former $b_{s+i}$), such that $d_i\in D_i$ exchanges symmetrically with $d_{i'}\in D_{i'}$ and matroids $N_i:=(M/d_i)\vert_{(D_1\setminus d_1)\cup\cdots\cup (D_{n(r-1,k)}\setminus d_{n(r-1,k)})}$ are the same (as restrictions of matroids $M_j$).

Consider $n(r-1,k)$-matroid $N_i$ of rank $r-1$ with $n(r-1,k)$ disjoint bases $F_1=D_1\setminus d_1,\dots,F_{n(r-1,k)}=D_{n(r-1,k)}\setminus d_{n(r-1,k)}$ (here we use the fact that $d_i\in D_i$ exchanges symmetrically with $d_{i'}\in D_{i'}$). From the inductive assumption it follows that there are disjoint bases $F'_1,\dots,F'_{n(r-1,k)}$ obtained from $F_1,\dots,F_{n(r-1,k)}$ by symmetric exchanges, among which there are $k$ bases (without loss of generality) $F'_{1},\dots,F'_{k}$ that contain the $k$-th blow-up of a basis of rank $r-1$. That is, one can label the elements of $F'_1\cup\cdots\cup F'_k$ with labels $l_1,\dots,l_{r-1}$, each $F'_{j}$ with distinct labels, such that every set of $r-1$ elements of distinct labels is a basis in $N_i$.

Notice that bases $B'_1=F'_1\cup d_1,\dots,B'_{n(r-1,k)}=F'_{n(r-1,k)}\cup d_{n(r-1,k)}$ are obtained from bases $D_1,\dots,D_{n(r-1,k)}$, hence also from bases $B_1,\dots,B_{n(r-1,k)}$, by symmetric exchanges in $M$. Moreover, bases $B'_{1},\dots,B'_{k}$ contain the $k$-th blow-up of a basis of rank $r$. Namely, one can label the elements of $B'_1\cup\cdots\cup B'_k$ with labels $l_1,\dots,l_r$ (we use the former labeling of $F'_1\cup\cdots\cup F'_k$, additionally elements $d_i$ get label $l_r$), each basis with distinct labels, such that every set of $r$ elements of distinct labels forms a basis in $M$. This proves the inductive assertion. 
\end{proof}

We are going to prove a generalization of Lemma \ref{LemmaBlowUp}, which asserts that additionally the desired $k$-th blow-up can be compatible with a fixed subset of a matroid. For this purpose we will use Ramsey theory for hypergraphs. A result of Erd\H{o}s \cite{Er64} implies the following lemma (see \cite{Du10} for possible generalizations):

\begin{lemma}\label{LemmaRamsey}
For every integers $r,k$ and $c$ there exists an integer $R=R(r,k,c)$, such that if $H$ is a $c$-colored (edges receive one of $c$ colors) complete $r$-uniform $r$-partite hypergraph of size $R$ (size of each part is $R$), then one can find in it a monochromatic complete subhypergraph $H'$ of size $k$.
\end{lemma} 

Let $B_1,\dots,B_k$ be disjoint bases of a matroid $M$ on the ground set $E$. Denote $F=E\setminus (B_1\cup\dots\cup B_k)$. Furthermore, we say that \emph{bases $B_1,\dots,B_k$ contain the $k$-th blow-up of a basis $B_1$ in $B_1\cup F$} if there exists a morphism $\psi:M\vert_{B_{1}\cup\dots\cup B_{k}\cup F}\rightarrow M\vert_{B_{1}\cup F}$, such that $\psi(B_1)=\dots=\psi(B_k)=B_1$ and $\psi$ is an identity on $B_1\cup F$. 

\begin{lemma}\label{LemmaBlowUp2}
For every positive integers $r,k$ and nonnegative integer $l$ there exists an integer $m=m(r,k,l)$, such that if $M$ is a matroid of rank $r$ on the ground set $E$, containing $m$ disjoint bases $B_1,\dots,B_m$ whose complement $F=E\setminus (B_1\cup\dots\cup B_m)$ is of size $l$, then there exists a modification of bases $B_1,\dots,B_m$ by symmetric exchanges to bases $B'_1,\dots,B'_m$ from which one can pick $k$ bases $B'_{i_1},\dots,B'_{i_k}$ that contain the $k$-th blow-up of basis $B_1$ in $B_1\cup F$. 
\end{lemma}

\begin{proof}
We will show that for $m=m(r,k,l)=n(r,R(r,k,2^{r+l}))$ the desired property holds, where $n,R$ are the functions from Lemmas \ref{LemmaBlowUp} and \ref{LemmaRamsey}.

Let $M$ be a matroid of rank $r$ on the ground set $E$, containing $m$ disjoint bases $B_1,\dots,B_m$ whose complement $F=E\setminus (B_1\cup\dots\cup B_m)$ has $l$ elements. Due to Lemma \ref{LemmaBlowUp} there exists a modification of bases $B_1,\dots,B_m$ by symmetric exchanges to bases $B'_1,\dots,B'_m$ from which one can pick $R:=R(r,k,2^{r+l})$ bases (without loss of generality) $B'_{1},\dots,B'_{R}$ that contain the $R$-th blow-up of basis $B'_{1}$. So, there exists a morphism $\psi:M\vert_{B'_{1}\cup\dots\cup B'_{R}}\rightarrow M\vert_{B'_{1}}$, such that $\psi(B'_1)=\dots=\psi(B'_R)=B'_1$ and $\psi$ is an identity on $B'_1$. 

There is only one possible extension of the morphism $\psi$ to $\psi':M\vert_{B'_{1}\cup\dots\cup B'_{R}\cup F}\rightarrow M\vert_{B'_{1}\cup F}$, such that $\psi'$ is an identity on $B'_1\cup F$. It is a morphism if for every $i=1,\dots,r-1$, for every $i$-element proper subset $S\subset B'_1$, and for every choice of representatives of sets $\{\psi^{-1}(b)\}_{b\in S}$, their union with a $(r-i)$-element subset $T\subset F$ is a basis of $M$ if and only if $S\cup T$ is a basis in $M$. We will show that, but for a smaller blow-up.

Consider a complete $r$-uniform $r$-partite hypergraph $H$ with parts $\psi^{-1}(e)$ for $e\in B'_1$, each part of size $R$. Edges of $H$ are bases in $M$, since $\psi$ is a morphism from $M$. Define a $2^{r+l}$-coloring $c$ of edges of $H$. For each $i$-element proper subset $S\subset B'_1$, for each $(r-i)$-element subset $T\subset F$ and for an edge $D$ of $H$, let the bit $c_{S,T}(D)$ of $c(D)$ be $1$ if $(\psi^{-1}(S)\cap D)\cup T$ is a basis in $M$, and $0$ otherwise. 

Using Lemma \ref{LemmaRamsey} for $H$ with coloring $c$ we get a monochromatic subhypergraph $H'$ of size $k$. Let $D_1,\dots,D_k$ be $k$ disjoint edges of $H'$ (they are also bases of $M$), and let $D_{k+1},\dots,D_R$ be disjoint edges of $H$ completing a partition of the ground set of $H$. Then, since $H$ is a complete hypergraph, bases $D_1,\dots,D_R$ are obtained from $B'_{1},\dots,B'_{R}$ by symmetric exchanges. Hence also bases $D_1,\dots,D_R,B'_{R+1},\dots,B'_m$ are obtained from bases $B_{1},\dots,B_{m}$ by symmetric exchanges. Directly from the construction, there exists a desired morphism $\psi:M\vert_{D_{1}\cup\dots\cup D_{k}\cup F}\rightarrow M\vert_{D_{1}\cup F}$.
\end{proof}

\section{White's conjecture for high degrees}\label{6}

\begin{theorem}\label{Theorem1}
For every $r$ there exists a constant $c(r)$ such that for every matroid $M$ of rank $r$ the homogeneous parts of degree at least $c(r)$ of the toric ideal $I_M$ are generated by quadratic binomials corresponding to symmetric exchanges.
\end{theorem}

\begin{proof}
Let $c(r)=(r+3)!+m(r,r(r+3)!,r(r+3)!)$, where $m$ is the function from Lemma \ref{LemmaBlowUp2}.

Let $M$ be a matroid of rank $r$. We have to show $(J_M)_d=(I_M)_d$ for $d\geq c(r)$. The inclusion $J_M\subset I_M$ implies that $(J_M)_d\subset (I_M)_d$ for every $d$. To prove the opposite inclusion, let $b\in I_M$ be a binomial of degree $d\geq c(r)$. By Theorem \ref{Theorem2} we have that $b=\sum_{i=1}^na_ib_i$, where $b_i\in I_M$ is a binomial of degree $(r+3)!$ and $a_i$ is a monomial of degree greater or equal to $m:=m(r,r(r+3)!,r(r+3)!)$. We show that $ab\in J_M$, for every binomial $b\in I_M$ of degree $(r+3)!$ and every monomial $a$ of degree $m$. 

Suppose that $b=y_{D_1}\cdots y_{D_{(r+3)!}}-y_{G_1}\cdots y_{G_{(r+3)!}}$ and $a=y_{B_1}\cdots y_{B_m}$.
Denote the union $D_1\cup\cdots\cup D_{(r+3)!}\cup B_1\cup\dots\cup B_m$ by $S$ as a set, and by $(S,\mu)$ as a multiset in which $\mu$ is the multiplicity function. Let $M'$ be a matroid obtained from $M\vert_S$ by replacing every element $e$ by $\mu(e)$ parallel elements. That is, there is a morphism $\psi': M'\rightarrow
M\vert_S$. Let $D'_1,\dots,D'_{(r+3)!},B'_1,\dots,B'_m$ be disjoint bases of $M'$ such that $\psi'(D'_i)=D_i$ and $\psi'(B'_i)=B_i$. Let $G'_1,\dots,G'_{(r+3)!}$ be disjoint bases of $M'$ such that $D'_1\cup\dots\cup D'_{(r+3)!}=G'_1\cup\dots\cup G'_{(r+3)!}$ and $\psi'(G'_i)=G_i$. Let $b',a'$ be the corresponding polynomials. Clearly, $b'\in I_{M'}$. Observe that $ab\in J_{M}$ if and only if $a'b'\in J_{M'}$. Therefore, we need to show that $a'b'\in J_{M'}$.

Due to Lemma \ref{LemmaBlowUp2} applied for $m$ disjoint bases $B'_1,\dots,B'_m$ in $M'$ whose complement $F=D'_1\cup\cdots\cup D'_{(r+3)!}$ is of size $l=r(r+3)!$, there exists a modification of bases $B'_1,\dots,B'_m$ by symmetric exchanges to bases $B''_1,\dots,B''_m$ among which $k=r(r+3)!$ bases (without loss of generality) $B''_1,\dots,B''_k$ contain the $k$-th blow-up of basis $B''_1$ in $B''_1\cup F$.
Let $\psi:M'\vert_{B''_{1}\cup\dots\cup B''_{k}\cup F}\rightarrow M'\vert_{B''_{1}\cup F}$ be a morphism, such that $\psi(B''_1)=\dots=\psi(B''_k)=B''_1$ and $\psi$ is an identity on $B''_1\cup F$. Denote $a''=y_{B''_1}\cdots y_{B''_m}$. Then, $a'-a''\in J_{M'}$. Hence, it is enough to show that $a''b'\in J_{M'}$, or even that $y_{B''_1}\cdots y_{B''_k}b'\in J_{M'}$.

We use Observation \ref{ObservationBlowDown} for morphism $\psi$ and $f=y_{B''_1}\cdots y_{B''_k}b'\in I_M$. We have $\psi(y_{B''_1}\cdots y_{B''_k}b')=y_{B''_1}^{r(r+3)!}b'$, so it is enough to show that $y_{B''_1}^{r(r+3)!}b'\in J_{M'}$.

Finally, we use \cite[Claim $4$]{LaMi14}. It asserts that for every basis $B$ of $M'$, if $b'\in I_{M'}$ is a binomial, then $y_B^{\deg_B(b)-\deg(b)}b'\in J_{M'}$ (where $\deg_B(y_{B'})=\left\vert B'\setminus B\right\vert$). In our case $y_{B''_1}^{r(r+3)!}b'\in J_{M'}$, since the degree of $b'$ is $(r+3)!$. This finishes the proof.
\end{proof}

\section{Acknowledgements}

We thank Professor J{\"u}rgen Herzog for a question during the conference in Osnabr\"{u}ck in $2015$ that initiated this study. We thank Jan Draisma and Mateusz Micha{\l}ek for stimulating discussions and helpful comments. We are also grateful to the anonymous referee for showing us the proof of Corollary \ref{Corollary2}.


\end{document}